\documentclass[english, 11pt]{amsart}

\usepackage{tikz}
\usepackage{aliascnt}
\usepackage{amsfonts}
\usepackage{mathrsfs}
\usepackage{amsthm}
\usepackage{amsmath,amssymb}
\usepackage[all,poly,knot]{xy}
\usepackage{amsfonts}
\usepackage{color}
\usepackage{hyperref}
\usepackage{comment}  
\usepackage{csquotes}

\theoremstyle{plain}
\newtheorem{thmx}{Theorem} 

\newtheorem{thm}{Theorem}[section]  

\newtheorem{cor}[thm]{{Corollary}} 

\newtheorem{lem}[thm]{{Lemma}}

\newtheorem{prop}[thm]{Proposition}

\theoremstyle{remark}
\newtheorem{rmk}[thm]{Remark}

\numberwithin{equation}{section}

\def\log{\mathrm{log}\,}

\def\Sym{\mathrm{Sym}}
\def\cH{\mathscr{H}}

\usepackage[top=1.5in, bottom=1.5in, left=1in, right=1in]{geometry}

 \usepackage[hyperpageref]{backref}
 
\usepackage{xcolor}
\hypersetup{
	colorlinks,
	linkcolor={red!50!black},
	citecolor={blue!50!black},
	urlcolor={blue!80!black}
}
\begin{document} 
\title[Second Main Theorem on Moduli Spaces]{Second Main Theorem on the Moduli Spaces of Polarized varieties} 

	\author{Ruiran Sun}
		 \address{Institut fur Mathematik, Universit\"at Mainz, Mainz, 55099, Germany}
	\email{ruirasun@uni-mainz.de}

\begin{abstract}
Let $(X,D)$ be a smooth log pair over $\mathbb{C}$ such that the complement $U := X \setminus D$ carries a maximally varied family of polarized manifolds. We prove a version of second main theorem on $(X,D)$ by using the Viehweg-Zuo construction of the family and McQuillan's tautological inequality. As an application, we generalize a classical result of Nadel about the distribution of entire curves in the (compactified) base space of polarized families. 
\end{abstract}


\maketitle
%\tableofcontents 

\section{Introduction}
Recently, Deng, Lu, Zuo and the author proved a big Picard type theorem for certain moduli spaces of polarized varieties \cite{DLSZ} via certain Nevanlinna theoretic arguments. Since the higher dimensional generalizations of the big Picard theorem follow naturally from the appropriate second main theorems/conjectures in Nevanlinna theory, it is natural to ask whether one can establish a second main theorem on these moduli spaces to deal with the more general case of holomorphic curves that intersect possibly the boundary of a moduli space.\\
We first recall the following conjecture of Vojta in Nevanlinna theory.
Let $(X,D)$ be a smooth log pair (i.e. $X$ is a smooth projective variety and $D$ is a normal crossing divisor of $X$) over $\mathbb{C}$. We fix an ample line bundle $\cH$ on $X$.
For any entire curve $f:\, \mathbb{C} \to X$ with $f(\mathbb{C}) \not\subset \mathrm{Supp}\,D$, Vojta's conjecture on the second main theorem {\cite[\S15]{voj11}} predicts the following inequality:
\[
T(r,f,K_X(D)) \leq_{\mathrm{exc}} N^{(1)}(r,f,D) + O(\log T(r,f,\cH) + \log r)
\]
where $T(r,f,K_X(D))$ is the {Nevanlinna characteristic function} of the log canonical line bundle $K_X(D)$ and $N^{(1)}(r,f,D)$ is the $1^{\mathrm{st}}$-truncated counting function of $D$. Here $\leq_{\mathrm{exc}}$ means that the estimate holds outside some exceptional set of $r \in \mathbb{R}_{>0}$ with finite Lebesgue measure.\\
More generally, let $B$ be a connected Riemann surface and $\sigma:\, B \to \mathbb{C}$ is a finite surjective holomorphic map, and we consider a holomorphic curve $f:\, B \to X$ with $f(B) \not\subset \mathrm{Supp}\,D$. Vojta's conjecture {\cite[Conjecture~27.5]{voj11}} predicts the following generalized inequality
\[
T(r,f,K_X(D)) \leq_{\mathrm{exc}} N^{(1)}(r,f,D) + N_{\mathrm{Ram}(\sigma)}(r) + O(\log T(r,f,\cH) + \log r).
\]
where $N_{\mathrm{Ram}(\sigma)}(r)$ is the ramification counting function for the ramified cover $\sigma$. The basics of Nevanlinna theory is reviewed in the beginning of section~\ref{tauto-ineq}.\\[.2cm]
In this notes, inspired by \cite{DLSZ} we consider these log pairs $(X,D)$ which can be interpreted as smooth compactification of the base space of a family. That is, we consider a smooth log pair $(X,D)$ over $\mathbb{C}$ such that the complement $U := X \setminus D$ carries a family of polarized manifolds.\\ 
We shall prove the following version of the second main theorem for $(X,D)$:

\begin{thmx}[= Theorem~\ref{SMT}]\label{thm-A}
Let $(X,D)$ be a smooth log pair over $\mathbb{C}$ and $U:= X \setminus D$. Suppose there is a smooth family $(\psi:\, V \to U, \mathcal{L})$ of polarized smooth varieties with semi-ample canonical sheaves and fixed Hilbert polynomials $h$, such that the induced classifying map from $U$ to the moduli scheme $\mathcal{M}_h$ is quasi-finite. Then for any holomorphic curve $f:\, B \to X$ which is not contained in the support of $D$, we have the following second main theorem: 
\[
T(r,f,K_X(D)) \leq_{\mathrm{exc}}c_{\psi}\cdot \left( N^{(1)}(r,f,D) +N_{\mathrm{Ram}(\sigma)}(r) \right) + O(\log T(r,f,\cH) + \log r).
\]
where $c_{\psi}$ is some positive constant which only depends on $(X,D)$ and the family $\psi$.
\end{thmx}
In section~\ref{VZ} we will recall that there is an ample line bundle $A$ over the base space $X$ which is closely related to the direct image sheaf of the family.
The second main theorem stated above actually follows from an inequality about the Nevanlinna characteristic function $T(r,f,A)$ of $A$:
\begin{thmx}[= Theorem~\ref{nevan-ineq}]
Let $(X,D)$ be the same as in Theorem~\ref{thm-A}. Then for any holomorphic curve $f:\, B \to X$ which is not contained in the support of $D$, we have
\[
T(r,f,A) \leq_{\mathrm{exc}} \frac{d+1}{2} \left( N^{(1)}(r,f,D) +N_{\mathrm{Ram}(\sigma)}(r) \right) + O(\log T(r,f,\cH) + \log r).
\]
where $d$ is the fiber dimension of the family $\psi$.  
\end{thmx}
This inequality should be regarded as an analytic version of the Arakelov type inequality of M\"oller-Viehweg-Zuo in \cite[Theorem~0.3]{MVZ-06}.\\

We have two applications. The first one generalizes a classical result of Nadel {\cite[Theorem~0.2]{nad89}}:
\begin{thmx}[= Corollary~\ref{nadel}]
Let $(X,D)$ be as in Theorem~\ref{thm-A}. Then for any nonconstant entire curve $f:\,\mathbb{C} \to X$ which ramifies over $D$ with order at least $c$, i.e. a positive constant which only depends on $(X,D)$ and the family $\psi$ such that $f^*D \geq c\cdot \mathrm{Supp}\,f^*D$, we have $f(\mathbb{C}) \subset D$.     
\end{thmx}
See Remark~\ref{rmk-nad} about how to remove the assumption on the entire curves by taking ramified covers of $X$ in the case of family of abelian varieties.\\[.2cm]
The second one gives an explicit constant for $c_{\psi}$ in Theorem~\ref{thm-A} in the case of Siegel modular varieties:
\begin{thmx}[= Corollary~\ref{siegel}]\label{thm-D}
Let $\mathcal{A}^{[n]}_g$ be the moduli space of principally polarized abelian varieties with level-$n$ structure ($n\geq 3$). Denote by $\bar{\mathcal{A}}^{[n]}_g$ the smooth compactification of $\mathcal{A}^{[n]}_g$ and $D:= \bar{\mathcal{A}}^{[n]}_g \setminus \mathcal{A}^{[n]}_g$ is the normal crossing boundary divisor. Let $f:\, B \to \bar{\mathcal{A}}^{[n]}_g$ be a holomorphic curve with $f(B) \not\subset \mathrm{Supp}\,D$. 
Then we have
\[
T(r,f,K_{\bar{\mathcal{A}}^{[n]}_g}(D)) \leq_{\mathrm{exc}} \frac{(1+g)^2}{2} \left( N^{(1)}(r,f,D) + N_{\mathrm{Ram}(\sigma)}(r) \right) + O(\log T(r,f,\cH) + \log r).
\]     
\end{thmx}

Note that in {\cite[Theorem~B]{RT18}} Rousseau and Touzet found a better constant for $c_{\psi}$ in the case of Hilbert modular varieties. For the $n$ dimensional Hilbert modular variety $(X,D)$ they proved that
\[
T(r,f,K_X(D)) \leq n\,N^{(1)}(r,f,D) + O(\log T(r,f,\cH) + \log r).
\] 
It is naturally to ask if one can improve these constants $c_{\psi}$ in Theorem~\ref{thm-A} and Theorem~\ref{thm-D}.\\

Just before this notes is finished, Brotbek and Brunebarbe uploaded a preprint \cite{BB20} on a similar theme of second main theorem type. Our setting and approach are different however: they consider base spaces which carry polarized VHS while we consider base spaces which carry families of polarized varieties; they use the Griffiths-Schmid metric and the logarithmic derivative lemma while we use the Finsler (pseudo)metric and the tautological inequality. The common feature is that we both use certain ample line bundle $A$ on $X$ and obtain an inequality of the Nevanlinna characteristic function of $A$ ( it is called the Arakelov-Nevanlinna inequality in \cite{BB20}, see Theorem~1.1 loc. cit.). In \cite{BB20} they used the Griffiths line bundle associated to the VHS, and we use the Viehweg line bundle $A$ associated to the family.\\

\noindent{\bf Acknowledgment.} This paper grew out of discussions about the second main theorem on moduli spaces with Professor Steven Shin-Yi Lu during the preparation of \cite{DLSZ}. I am very grateful for the inspiring discussions with him, as well as invaluable comments and suggestions he provided. I would like to sincerely thank Professor Kang Zuo for explaining his celebrated work with Viehweg to me, and for his constant supports and encouragements.

\section{Viehweg-Zuo construction and curvature current inequality}\label{VZ}
In this section we recall briefly the Viehweg-Zuo's construction and use it to obtain an inequality of curvature currents on the Riemann surface $B$.\\

We first define the open Riemann surface $\Sigma:= B \setminus f^*D$. So the restriction is a holomorphic curve mapping into the base space $\gamma:\, \Sigma \to U=X \setminus D$.

By the theory of Viehweg-Zuo (cf. {\cite[\S6]{VZ-1}} or {\cite[\S4]{VZ-2}}), for the smooth log pair $(X,D)$ as in Theorem~\ref{thm-A}, we have the following geometric objects over $X$: an ample line bundle $A$ whose restriction on the smooth locus $U$ is isomorphic to the Viehweg line bundle $\mathrm{det}(\psi_*\omega^{\mu}_{Y/X})^{\nu}$ (cf. {\cite[\S4]{VZ-1}}), the deformation Higgs bundel $(F,\tau)$ associated to the family $g$, a logarithmic Hodge bundle $(E,\theta)$ over $X$ with poles along $D+T$ ($T$ is some normal crossing divisor), and the comparison maps $\rho$ which fits into the following commutative diagram
\begin{align}
  \label{eq:3}
\xymatrixcolsep{5pc}\xymatrix{
F^{p,q} \ar[r]^-{\tau^{p,q}} \ar[d]^{\rho^{p,q}} & F^{p-1,q+1} \otimes \Omega^1_{X}(\log D) \ar[d]^{\rho^{p-1,q+1} \otimes \iota} \\
A^{-1} \otimes E^{p,q} \ar[r]^-{\mathrm{id} \otimes \theta^{p,q}} &  A^{-1} \otimes E^{p-1,q+1} \otimes \Omega^1_{X}(\log (D+T)).
} 
\end{align}
We iterate the Higgs maps $\tau^{p,q}$ to get
\[
\tau^{d-q+1,q-1} \circ \cdots \circ \tau^{d,0} :\, F^{d,0} \to F^{d-q,q} \otimes \bigotimes^q \Omega^1_{X}(\log D).
\]
where $d$ is the fiber dimension of the family $g$.
This composition  factors through
\[
\tau^q:\, F^{d,0} \to F^{d-q,q} \otimes \Sym^q\, \Omega^1_{X}(\log D)
\]
because the Higgs field $\tau$ satisfies $\tau \wedge \tau=0$. As $\mathcal{O}_X$ is a subsheaf of $F^{d,0}$, the composition of maps
\[
\xymatrix{
\Sym^q \, T_{X}(-\log D) \ar[r]^-{\subset} & F^{d,0} \otimes \Sym^q \, T_X(-\log D) \ar[r]^-{\tau^q \otimes \mathrm{id}}  &  F^{d-q,q} \otimes \Sym^q\, \Omega^1_X(\log D) \otimes \Sym^q \, T_X(-\log D) \ar[d]^{\mathrm{id} \otimes <,>} \\
  &  &  F^{d-q,q}
}
\]
makes sense, which we will still denote as $\tau^q$ by abuse of notations.
It is a remarkable fact that 
\[
\rho^{d-1,1}\circ \tau^1:\, T_{X}(-\log D) \to A^{-1} \otimes E^{d-1,1}
\]
is generically injective.\\

Now we put the holomorphic curve $\gamma:\, \Sigma \to U$ into the picture. 
We first note that the pull back of $(F,\tau)$ induces a new (holomorphic) Higgs bundle over $\Sigma$:
\[
(F_{\gamma},\tau_{\gamma})\quad \textrm{with $F_{\gamma}:= \gamma^*F$ and $\tau_{\gamma}:\, F_{\gamma} \xrightarrow{\gamma^*\tau} F_{\gamma} \otimes \gamma^*\Omega^1_X(\log D) \xrightarrow{\mathrm{id} \otimes d\gamma} F_{\gamma} \otimes \Omega^1_{\Sigma}$}.
\]
We define $(E_{\gamma},\theta_{\gamma})$ in the same way, emphasizing that $\theta_{\gamma}$ has log poles along $\gamma^*T$. Then the commutative diagram (\ref{eq:3}) also holds on $\Sigma$:
\begin{align}
  \label{eq:4}
\xymatrixcolsep{5pc}\xymatrix{
F^{p,q}_{\gamma} \ar[r]^-{\tau^{p,q}_{\gamma}} \ar[d]^{\rho^{p,q}_{\gamma}} & F^{p-1,q+1}_{\gamma} \otimes \Omega^1_{\Sigma} \ar[d]^{\rho^{p-1,q+1}_{\gamma} \otimes \iota} \\
A^{-1}_{\gamma} \otimes E^{p,q}_{\gamma} \ar[r]^-{\mathrm{id} \otimes \theta^{p,q}_{\gamma}} &  A^{-1}_{\gamma} \otimes E^{p-1,q+1}_{\gamma} \otimes \Omega^1_{\Sigma}(\log \gamma^*T)
} 
\end{align}
where $A_{\gamma}$ denotes $\gamma^*A$. Similarly, we can define the iteration of Higgs maps on $\Sigma$
\[
\tau^q_{\gamma}:\, T^{\otimes q}_{\Sigma} \to F^{d-q,q}_{\gamma}.
\]

Now we shall define a sub Higgs bundle of the holomorphic Hodge bundle $(E_{\gamma},\theta_{\gamma})$ by using the iterations of Higgs maps. For each $q \in \mathbb{Z}_{\geq 0}$, we define $G^{d-q,q}$ as the \emph{saturation} of the image of 
\[
A \otimes T^{\otimes q}_{\Sigma} \xrightarrow{\mathrm{id} \otimes \tau^q_{\gamma}} A \otimes F^{d-q,q}_{\gamma} \to E^{d-q,q}_{\gamma}
\]
in $E^{d-q,q}_{\gamma}$.

\begin{lem}
$\theta^{d-q,q}_{\gamma}(G^{d-q,q}) \subset G^{d-q-1,q+1} \otimes \Omega^1_{\Sigma}$.  
\end{lem}
\begin{proof}
By (\ref{eq:4}) we have the following commutative diagram
\[
\xymatrix{
A \otimes T^{\otimes q}_{\Sigma} \ar[r] \ar[rd] & A \otimes F^{d-q,q}_{\gamma} \ar[r] \ar[d] &  E^{d-q,q}_{\gamma} \ar[d] \\
   &  A \otimes F^{d-q-1,q+1}_{\gamma} \otimes \Omega^1_{\Sigma} \ar[r] & E^{d-q-1,q+1}_{\gamma} \otimes \Omega^1_{\Sigma}(\log \gamma^*T).
}
\] 
By the construction of $\tau^q_{\gamma}$, we know that the image of 
\[
A \otimes T^{\otimes q}_{\Sigma} \to A \otimes F^{d-q-1,q+1}_{\gamma} \otimes \Omega^1_{\Sigma} \xrightarrow{\rho^{d-q-1,q+1}_{\gamma}\otimes \mathrm{id}} E^{d-q-1,q+1}_{\gamma} \otimes \Omega^1_{\Sigma}
\]
is exactly $G^{d-q-1,q+1} \otimes \Omega^1_{\Sigma}$. Then the statement follows from the commutativity of the diagram above.
\end{proof}

Therefore, $(G,\theta_G):= (\bigoplus G^{p,q},\theta_{\gamma}|_G)$ is a sub Higgs bundle of $(E_{\gamma},\theta_{\gamma})$.
\begin{prop}\label{curv-G}
Let $\mathrm{det}\,G$ be the determinant bundle of $G$ and $h$ be the hermitian metric on it induced by the Hodge metric on $E$. Then 
\[
\Theta(\mathrm{det}\,G,h) \leq 0.
\]  
\end{prop}
\begin{proof}
Note that $\mathrm{det}(G,\theta_G) \subset \bigwedge^r(E_{\gamma},\theta_{\gamma})$, where $r$ is the rank of $G$. Since $\theta_G$ is nilpotent, we know that $\mathrm{det}(G,\theta_G)= (\mathrm{det}\,G,0)$. In particular, by the compatibility of Higgs maps, it follows that $\mathrm{det}\,G$ lands in the kernel of the Higgs map of $\bigwedge^r(E_{\gamma},\theta_{\gamma})$.\\
Now suppose that $(E_{\gamma},\theta_{\gamma})$ is the Hodge bundle associated to a polarized variation of Hodge structures $\mathbb{V}$ over the Riemann surface $\Sigma$. Then $\bigwedge^r(E_{\gamma},\theta_{\gamma})$ is the Hodge bundle associated to the polarized variation of Hodge structures $\bigwedge^r\mathbb{V}$ by functoriality. That means we can use Griffiths’ curvature computation for $\bigwedge^r(E_{\gamma},\theta_{\gamma})$ and conclude that the curvature form of $\mathrm{det}\,G$ is semi-negative.
\end{proof}

We shall use McQuillan's tautological inequality to derive the second main theorem. In order to do so, we need to work on the projective (log) tangent bundle of $(X,D)$
\[
\mathbb{P}(T_X(-\log D)) := \mathrm{Proj}\,\Sym^{\bullet} \Omega^1_X(\log D).
\]
The iterations of Higgs maps $\tau^q$ can be lifted to 
\begin{align}
\tilde{\tau}^q:\, \mathcal{O}(-q) \to \pi^*F^{d-q,q}
\end{align}
where $\mathcal{O}(-q)$ is the $q$-th power of the tautological line bundle $\mathcal{O}(-1)$ and $\pi:\,\mathbb{P}(T_X(-\log D)) \to X$ is the projection.\\

Now we consider the lifting of $\gamma$
\[
\xymatrix{
  & \mathbb{P}(T_X(-\log D)) \ar[d]^{\pi} \\
\Sigma \ar[ur]^{\gamma'} \ar[r]^{\gamma} & X
}
\]
induced by the tangent map of $\gamma$. Then the tangent map $d\gamma$ naturally factors through
\[
T_{\Sigma} \to \gamma'^*\mathcal{O}(-1) \to \gamma^*T_X(-\log D).
\]
As a consequence, the iterations of Higgs maps $\tau^q_{\gamma}$ on $\Sigma$ also factors through
\[
T^{\otimes q}_{\Sigma} \to \gamma'^*\mathcal{O}(-q) \xrightarrow{\gamma'^*\tilde{\tau}^q} F^{d-q,q}_{\gamma}.
\]
Thus the sub Higgs bundle $G^{d-q,q}$ is also the saturation of the image of 
\[
A_{\gamma} \otimes \gamma'^*\mathcal{O}(-q) \xrightarrow{\mathrm{id} \otimes \gamma'^*\tilde{\tau}^q} A_{\gamma} \otimes F^{d-q,q}_{\gamma} \to E^{d-q,q}_{\gamma}.
\]
Then we have maps
\begin{align}\label{zeta}
\zeta^q:\,\gamma'^*\mathcal{O}(-q) \to A^{-1}_{\gamma} \otimes G^{d-q,q}
\end{align}
for $q =0,1,\dots,d$. Note that $\rho^{d-1,1}_{\gamma} \circ \tau^1_{\gamma}$ is nonzero implies that $\zeta^1$ is nonzero. So there exists a positive integer $m$ such that $\zeta^m \neq 0$ and $\zeta^{m+1} = 0$, which is called the maximal length of iteration. Obviously $m \leq d$ and $\mathrm{det}\, G = \bigotimes^{m}_{q=0}G^{d-q,q}$ ( note that $G^{d-q,q} = 0$ for $q >m$ ).\\

In the rest of this section we shall use those nonzero maps $\zeta^q$ to relate the curvature currents of the ample line bundle $A$ and the tautological line bundle $\mathcal{O}(-1)$.\\
We first note that for each $q$ we can construct a (pseudo) metric $F_q$ on $\mathcal{O}(-1)$ via the following composition map 
\[
\mathcal{O}(-q) \xrightarrow{\tilde{\tau}^q} \pi^*F^{d-q,q} \to \pi^*(A^{-1} \otimes E^{d-q,q}).
\]
More precisely, we pull back the product metric $\pi^*(h_{A^{-1}} \otimes h^{d-q,q})$ to $\mathcal{O}(-q)$ and take the $q$-th root. Here $h_{A^{-1}}$ is the Fubini-Study metric on $A^{-1}$ and $h^{d-q,q}$ is the restriction of the Hodge metric on $E^{d-q,q}$.\\
Note that $F_q$ is a bounded pseudometric with possible degeneration on $\mathbb{P}(T_X(-\log D))$. We denote by $c_1(\mathcal{O}(-1),F_q)$ the curvature current with respect to $F_q$, i.e. locally it can be written as
\[
dd^c\log \Vert \sigma \Vert^{-2}_{F_q}
\]
for some non-vanishing section $\sigma$. Let $c_1(\mathcal{O}(-1))$ be the usual curvature (1,1)-form with respect to some smooth metric $F$ of $\mathcal{O}(-1)$. Then we have $F_q = \phi_q \cdot F$ and thus
\[
c_1(\mathcal{O}(-1)) = c_1(\mathcal{O}(-1),F_q) + dd^c\log \phi_q
\]
for some bounded function $\phi_q$. Then we have
\begin{prop}
For any $q$ such that $\zeta^q$ is a nonzero map, we have the current inequality on $\Sigma$:
\begin{align}\label{current-ineq-1}
q \cdot \gamma'^* c_1(\mathcal{O}(-1)) \leq \gamma^*c_1(A^{-1}) + \Theta(G^{d-q,q},h) + q \,\gamma'^*dd^c\log \phi_q.
\end{align}
\end{prop}

Using the curvature property of $\mathrm{det}\,G$, we can obtain the following curvature inequality, which is crucial for our main result.

\begin{prop}[curvature currents inequality]
Let $m$ be the maximal length of iteration, namely the largest integer so that $\zeta^m \not\equiv 0$. Then we have the following inequality of currents
\begin{align}\label{current-ineq-sum}
\frac{m+1}{2} f'^* c_1(\mathcal{O}(-1)) \leq f^*c_1(A^{-1}) + \frac{1}{m} \sum^m_{q=1}q \,f'^*dd^c\log \phi_q
\end{align}
on $B$. Here $f':\, B \to \mathbb{P}(T_X(-\log D))$ is the meromorphic map induced by the tangent map.  
\end{prop}
\begin{proof}
By (\ref{current-ineq-1}) we have 
\begin{align}
  \label{current-ineq-1-1}
q \cdot \gamma'^* c_1(\mathcal{O}(-1)) -\gamma^*c_1(A^{-1}) - q \,\gamma'^*dd^c\log \phi_q \leq  \Theta(G^{d-q,q},h).  
\end{align}
Summing up (\ref{current-ineq-1-1}) for $q$ from $1$ to $m$, we get
\begin{align}\label{current-lhs}
\frac{(m+1)m}{2} f'^* c_1(\mathcal{O}(-1)) - m\, f^*c_1(A^{-1}) -  \sum^m_{q=1}q \,f'^*dd^c\log \phi_q \leq \Theta(\mathrm{det}\,G,h) \leq 0   
\end{align}
on $\Sigma$. Since $\Theta(\mathrm{det}\,G,h)$ as a current on $B$ might have positive factor supported on $f^*D$ a priori, we need to study the asymptotic behavior of $\Theta(\mathrm{det}\,G,h)$ near the boundary $f^*D$.\\[.2cm]
Recall that $(G,\theta_G)$ is a sub Higgs bundle ( without log poles along $f^*T$ ) of the Hodge bundle $(E_{\gamma},\theta_{\gamma})$. By Borel's theorem the monodromy of $(E,\theta)$ along each irreducible components of $D$ is quasi-unipotent. Thus we can find a finite ramified covering $\delta:\, B' \to B$ such that the pull-back of $(E_{\gamma},\theta_{\gamma})$ on $B'$ has unipotent monodromy around each point in $(\delta \circ f)^*D$. Then by Schmid's estimate of the Hodge norm {\cite[Theorem~6.6]{sch73}} we know that the Hodge metric $h'$ on $\delta^*\mathrm{det}\,G$ is \emph{good} in the sense of Mumford {\cite[\S1]{mum77}} at each point of $(\delta \circ f)^*D$. So Mumford's theorem implies that its Chern form $\Theta(\delta^*\mathrm{det}\,G,h')$ is integrable on $B'$ (see also {\cite[\S3]{pet84}} for details). Hence we have
\[
\frac{(m+1)m}{2} (\delta\circ f')^* c_1(\mathcal{O}(-1)) - m\, (\delta\circ f)^*c_1(A^{-1}) -  \sum^m_{q=1}q \,(\delta \circ f')^*dd^c\log \phi_q \leq \Theta(\delta^*\mathrm{det}\,G,h') \leq 0
\]
on $B'$.
Now the inequality (\ref{current-ineq-sum}) holds on $B$ since its pull back holds on $B'$.
\end{proof}

\section{Tautological Inequality and Second Main Theorem}\label{tauto-ineq}

In this section we shall use McQuillan's tautological inequality to show our second main theorem. We first recall some definitions in Nevanlinna theory.\\

Let $X$ be a smooth projective variety over $\mathbb{C}$. Let $L$ be a line bundle over $X$ which is equipped with a smooth hermitian metric. We denote by $c_1(L)$ the Chern form of $L$ with respect to this metric. Then for any holomorphic curve $f:\, B \to X$ ($B$ is a finite ramified cover of $\mathbb{C}$ as in the introduction ) we define the \emph{Nevanlinna characteristic function}:
\[
T(r,f,L) := \int^r_0 \frac{dt}{t} \int_{B[t]}f^*c_1(L) 
\]
where $B[t]:= \sigma^{-1}(\mathbb{D}(t))$.\\
For an effective divisor $D$ of $X$ we define the \emph{counting function} for holomorphic curves $f:\, B \to X$ with $f(B) \not\subset \mathrm{Supp}\,D$: 
\[
N(r,f,D) := \int^r_0 \left( \sum_{p \in B[t]} \mathrm{ord}_pf^*D \right)\frac{dt}{t}.
\]
The ramification counting function is just the counting function for the ramification divisor of $\sigma:\, B \to \mathbb{C}$:
\[
N_{\mathrm{Ram}(\sigma)}(r) := \int^r_0 \left( \sum_{p \in B[t]} \mathrm{ord}_p\mathrm{Ram}(\sigma) \right)\frac{dt}{t}.
\]
Moreover, for an integer $k \geq 1$ one can define the $k$-th truncated counting function:
\[
N^{(k)}(r,f,D) := \int^r_0 \left( \sum_{p \in B[t]} \mathrm{min}\{k, \mathrm{ord}_pf^*D \} \right)\frac{dt}{t}.
\]
One has the following classical result:
\begin{prop}[Nevanlinna Inequality]
\[
N(r,f,D) \leq T(r,f,\mathcal{O}_X(D)) + O(1).
\]  
\end{prop}
Here the notation $O(1)$ denotes a bounded function. Proof of the Nevanlinna Inequality can be found in standard textbooks of Nevanlinna theory, for instance \cite{NW14}.\\

Now we recall the tautological inequality.\\
We will follow the notations of Vojta {\cite[\S29]{voj11}}. 
As mentioned before, the tangent map of $\gamma:\, \Sigma \to U$ induces a holomorphic map $\gamma':\, \Sigma \to \mathbb{P}(T_X(-\log D))$. Denote by $f':\, B \to \mathbb{P}(\Omega_X(\log D))$ the associated \emph{meromorphic map}. One can also define the Nevanlinna characteristic function as above
\[
T(r,f',\mathcal{O}(1)) := \int^r_0 \frac{dt}{t} \int_{B[t]} f'^*c_1(\mathcal{O}(1)).
\]
Then we have:
\begin{thm}[Tautological inequality, \cite{mc98}]
\[
T(r,f',\mathcal{O}(1)) \leq_{\mathrm{exc}} N^{(1)}(r,f,D) + N_{\mathrm{Ram}(\sigma)}(r) + O(\log T(r,f,\cH) + \log r).
\] 
\end{thm}

Now we can prove our main result.

\begin{thm}\label{nevan-ineq}
\[
T(r,f,A) \leq_{\mathrm{exc}} \frac{d+1}{2} \left(N^{(1)}(r,f,D) + N_{\mathrm{Ram}(\sigma)}(r)\right) + O(\log T(r,f,\cH) + \log r).
\]  
\end{thm}

\begin{proof}
Taking the integration $\int^r_0 \frac{dt}{t} \int_{B[t]}\bullet$ on both sides of (\ref{current-ineq-sum}), we get
\[
-\frac{m+1}{2}\, T(r,f',\mathcal{O}(1)) \leq -T(r,f,A) + \frac{1}{m} \sum^m_{q=1}q \,\int^r_0 \frac{dt}{t} \int_{B[t]} dd^c\log f'^*\phi_q.
\]
Notice that
\[
\int^r_0 \frac{dt}{t} \int_{B[t]} dd^c\log f'^*\phi_q = O(\log r)
\]
by the Green-Jensen formula {\cite[Lemma~1.1.5]{NW14}} and the boundedness of $\phi_q$. Thus we have
\[
T(r,f,A) \leq  \frac{m+1}{2}\, T(r,f',\mathcal{O}(1)) + O(\log r) \leq \frac{d+1}{2}\, T(r,f',\mathcal{O}(1)) + O(\log r).
\]
Now use the tautological inequality.
\end{proof}

Since $A$ is ample, one can always find a positive integer $k$ which \emph{only depends on $K_X+D$ and the family $\psi$} such that 
\[
kA - (K_X+D) \geq 0.
\]
Define $c_{\psi} := \frac{k(d+1)}{2}$. Then we get:
\begin{thm}[Second Main Theorem]\label{SMT}
\[
T(r,f,K_X(D)) \leq_{\mathrm{exc}} c_{\psi} \cdot \left( N^{(1)}(r,f,D) + N_{\mathrm{Ram}(\sigma)}(r) \right) + O(\log T(r,f,\cH) + \log r).
\]  
\end{thm}
With the knowledge and results we have from the Viehweg-Zuo construction, it should be possible to give good explicit bounds on the integer $k$ above and hence on the constant $c_{\psi}$.

\section{Applications}
In this section we make some applications of Theorem~\ref{nevan-ineq}.
The first application is a generalization of a classical result due to Nadel {\cite[Theorem~0.2]{nad89}}.
\begin{cor}\label{nadel}
Let $(X,D)$ be the same as in Theorem~\ref{thm-A}. Then for any nonconstant entire curve $f:\,\mathbb{C} \to X$ which ramifies over $D$ with order at least $c$, a positive constant which only depends on $(X,D)$ and the family $\psi$ such that $f^*D \geq c\cdot \mathrm{Supp}\,f^*D$, we have $f(\mathbb{C}) \subset D$.   
\end{cor}
\begin{proof}
We follow Nadel's strategy.
Suppose $f(\mathbb{C}) \not\subset D$. Then we can use the inequality obtained in Theorem~\ref{nevan-ineq}.\\
Since $A$ is ample, we can find an integer $k$ such that $A^{k}(-3D)$ is also ample.\\  
Now we take $c:= k(d+1)/2$.
Note that the assumption on the entire curve $f$ implies that 
\[
f^*D \geq \frac{k(d+1)}{2} \mathrm{Supp}\,f^*D.
\] 
Then by Theorem~\ref{nevan-ineq} we have
\[
T(r,f,A^k) \leq_{\mathrm{exc}} \frac{k(d+1)}{2} N^{(1)}(r,f,D) + O(\log T(r,f,\cH) + \log r) \leq N(r,f,D) + O(\log T(r,f,\cH) + \log r) 
\]
for $r \gg 0$.\\
By the first main theorem we know that $N(r,f,D) \leq T(r,f,D) + O(1)$. Thus we have
\[
T(r,f,A^k) \leq_{\mathrm{exc}} T(r,f,D) + O(\log T(r,f,\cH) + \log r).
\]
Dividing both side by $T(r,f,A^k)$, we get
\[
1 \leq \frac{T(r,f,D)}{T(r,f,A^k)} + O \left( \frac{\log T(r,f,\cH)}{T(r,f,A^k)}  + \frac{\log r}{T(r,f,A^k)} \right) \leq \frac{1}{2}   + c\cdot \frac{\log r}{T(r,f,A^k)}
\]
for some positive constant $c$, and $r \gg 0$. Therefore we get
\[
T(r,f,A^k) \leq 2c \, \log r
\]
for $r \gg 0$. But this is impossible since we can always replace $f:\,\mathbb{C} \to X$ by $\mathbb{C} \xrightarrow{\mathrm{exp}} \mathbb{C} \xrightarrow{f} X$.
\end{proof}
\begin{rmk}\label{rmk-nad}
Corollary~\ref{nadel} can be regarded as a generalization of Nadel's theorem because in the case that $U= X \setminus D$ carries a family of principally polarized abelian varieties, one can always find a covering $\eta:\, X' \to X$ which is unramified over $U$ and has sufficiently large ramification order over $D$ (cf. the proof of {\cite[Proposition~4.4]{mum77}}). Now for any entire curve $f:\, \mathbb{C} \to X'$, the composed entire curve $\eta \circ f$ satisfies the assumption in Corollary~\ref{nadel} and one conclude that $f(\mathbb{C}) \subset D'=(\eta^*D)_{\mathrm{red}}$.    
\end{rmk}

Our second application is about the second main theorem on the Siegel modular varieties.
\begin{cor}\label{siegel}
Let $\mathcal{A}^{[n]}_g$ be the moduli space of principally polarized abelian varieties with level-$n$ structure ($n\geq 3$). Denote by $\bar{\mathcal{A}}^{[n]}_g$ the smooth compactification of $\mathcal{A}^{[n]}_g$ and $D:= \bar{\mathcal{A}}^{[n]}_g \setminus \mathcal{A}^{[n]}_g$ is the normal crossing boundary divisor. Let $f:\, B \to \bar{\mathcal{A}}^{[n]}_g$ be a holomorphic curve with $f(B) \not\subset \mathrm{Supp}\,D$. 
Then we have
\[
T(r,f,K_{\bar{\mathcal{A}}^{[n]}_g}(D)) \leq_{\mathrm{exc}} \frac{(1+g)^2}{2} \left( N^{(1)}(r,f,D) + N_{\mathrm{Ram}(\sigma)}(r) \right) + O(\log T(r,f,\cH) + \log r).
\]   
\end{cor}
\begin{proof}
Since $n \geq 3$, we know that $\mathcal{A}^{[n]}_g$ is a fine moduli space and thus carries a universal family. In this situation the (log) canonical divisor of the Siegel modular variety is related to the variations of Hodge structures of the universal family.\\

Denote by $\mathbb{V}$ the associated polarized VHS of weight one and $(E=E^{1,0} \oplus E^{0,1},\theta= \theta^{1,0} \oplus \theta^{0,1})$ is the Hodge bundle of it. We use the same notation for Deligne's canonical extension of it. Thus $(E,\theta)$ is a graded logarithmic Higgs bundle on $\bar{\mathcal{A}}^{[n]}_g$ with $\mathrm{rk}\,E^{1,0} =g$. Furthermore, we have the isomorphism
\[
\Omega^1_{\bar{\mathcal{A}}^{[n]}_g}(D) \cong \Sym^2 E^{1,0}
\]
and thus
\[
K_{\bar{\mathcal{A}}^{[n]}_g}(D) \cong \mathrm{det}\,\Sym^2 E^{1,0} = (\mathrm{det}\,E^{1,0})^{\otimes {g+2-1 \choose g}} = (E^{g,0})^{\otimes g+1}
\]
where $E^{g,0} = \bigwedge^gE^{1,0}$ is the first Hodge bundle associated to the VHS of weight-$g$ (cf. {\cite[\S2]{falt83}} or {\cite[II.1]{HS02}}).\\

Now we shall apply our main theorem~\ref{nevan-ineq}. Note that in the case of families of abelian varieties the Viehweg line bundle $A$ is nothing but the first Hodge bundle $E^{g,0}$ ( $F^{p,q} = R^q\psi_*(\Omega^p_{Y/X}(\log \Delta) \otimes \omega^{-1}_{X/Y}) \cong R^q\psi_*(\Omega^p_{Y/X}(\log \Delta) ) \otimes \psi_*\omega^{-1}_{X/Y} = E^{p,q} \otimes (E^{g,0})^{-1}$ ). Thus we have
\[
T(r,f,E^{g,0}) \leq_{\mathrm{exc}} \frac{g+1}{2} \left(N^{(1)}(r,f,D) + N_{\mathrm{Ram}(\sigma)}(r)\right) + O(\log T(r,f,\cH) + \log r).
\] 
On the other hand, $K_{\bar{\mathcal{A}}^{[n]}_g}(D) \cong (E^{g,0})^{\otimes g+1}$. Those two facts give us the inequality in the statement.
\end{proof}

%\bibliography{Second-Main-Theorem}
%\bibliographystyle{amsalpha}

\end{document}